\documentclass[a4paper]{article}
\usepackage{hyperref}
\hypersetup{hypertexnames = false, bookmarksdepth = 2, bookmarksopen = true, colorlinks, linkcolor = black, citecolor = black, urlcolor = black, pdfstartview={XYZ null null 1}}

\usepackage{amsfonts}
\usepackage[fleqn, leqno]{amsmath}
\usepackage{amsthm}

\usepackage{cleveref}

\usepackage{booktabs}
\usepackage{diagbox}
\usepackage{enumitem}
\usepackage{mathtools}
\usepackage{parskip}
\usepackage{thmtools}
\usepackage{tikz-cd}
\usepackage[colorinlistoftodos, textsize = footnotesize]{todonotes}
\usepackage{xparse}
\usepackage{xspace}

\usepackage[utf8]{inputenc}
\usepackage[T1]{fontenc}
\usepackage{libertine}
\usepackage[libertine]{newtxmath}
\usepackage[scaled=0.83]{beramono}
\usepackage{eucal}
\usepackage{microtype}
\usepackage[all]{xy}

\newcounter{todocounter}
\DeclareDocumentCommand\addreference{g}{\stepcounter{todocounter}\todo[color = blue!30]{\thetodocounter. Add reference\IfNoValueF{#1}{: #1}}\xspace}
\DeclareDocumentCommand\checkthis{g}{\stepcounter{todocounter}\todo[color = red!50]{\thetodocounter. Check this\IfNoValueF{#1}{: #1}}\xspace}
\DeclareDocumentCommand\fixthis{g}{\stepcounter{todocounter}\todo[color = orange!50]{\thetodocounter. Fix this\IfNoValueF{#1}{: #1}}\xspace}
\DeclareDocumentCommand\expand{g}{\stepcounter{todocounter}\todo[color = green!50]{\thetodocounter. Expand\IfNoValueF{#1}{: #1}}\xspace}

\declaretheoremstyle[
  spaceabove = 3pt,
  spacebelow = 3pt,
]{lecture}
\theoremstyle{lecture}
\newtheorem{theorem}{Theorem}
\newtheorem{corollary}[theorem]{Corollary}
\newtheorem{definition}[theorem]{Definition}

\newtheorem{lemma}[theorem]{Lemma}
\newtheorem{proposition}[theorem]{Proposition}
\newtheorem{remark}[theorem]{Remark}

\newtheorem{alphatheorem}{Theorem}

\crefname{alphatheorem}{theorem}{theorems}

\crefname{alphaproposition}{proposition}{propositions}

\newtheorem{alphacorollary}[alphatheorem]{Corollary}
\crefname{alphacorollary}{corollary}{corollaries}

\makeatletter
\def\gitfootnote{\gdef\@thefnmark{}\@footnotetext}
\makeatother

\mathchardef\mhyphen="2D
\newcommand\dash{\nobreakdash-\hspace{0pt}}

\newcommand\hilbn[2]{\ensuremath{#2^{[#1]}}} 
\newcommand\nestedhilbn[3]{\ensuremath{#3^{[#1,#2]}}} 

\newcommand\alt{\ensuremath{\mathfrak{a}}}
\newcommand\bounded{\ensuremath{\mathrm{b}}}
\newcommand\FF{\ensuremath{\mathsf{F}}}

\newcommand\II{\ensuremath{\mathsf{I}}}

\newcommand\LL{\ensuremath{\mathrm{L}}}
\newcommand\LLL{\ensuremath{\mathbf{L}}}
\newcommand\RR{\ensuremath{\mathrm{R}}}
\newcommand\RRR{\ensuremath{\mathbf{R}}}
\newcommand\serre{\ensuremath{\mathbb{S}}}
\newcommand\sym{\ensuremath{\mathfrak{S}}}

\DeclareMathOperator\im{im}
\DeclareMathOperator\Bl{Bl}
\DeclareMathOperator\coh{coh}
\DeclareMathOperator\cone{cone}
\DeclareMathOperator\derived{\mathbf{D}}
\DeclareMathOperator\Ext{Ext}
\DeclareMathOperator\sExt{\mathcal{E}xt}
\DeclareMathOperator\FM{\Phi}
\DeclareMathOperator\HH{H}
\DeclareMathOperator\hh{h}
\DeclareMathOperator\Hom{Hom}

\DeclareMathOperator\RRRsHom{\mathbf{R}\mathcal Hom}
\DeclareMathOperator\identity{id}
\DeclareMathOperator\image{im}

\DeclareMathOperator\AJ{AJ}
\DeclareMathOperator\Jac{Jac}
\DeclareMathOperator\moduli{M}
\DeclareMathOperator\rk{rk}
\DeclareMathOperator\Sym{Sym}
\DeclareMathOperator\topdeg{topdeg}

\DeclareMathOperator\Aut{Aut}
\DeclareMathOperator\MM{\mathsf M}
\DeclareMathOperator\Res{Res}
\DeclareMathOperator\id{id}
\DeclareMathOperator\Pic{Pic}
\DeclareMathOperator\Sing{Sing}

\title{Derived categories of (nested) Hilbert schemes}
\author{Pieter Belmans \and Andreas Krug}

\begin{document}

\maketitle

\begin{abstract}
  In this paper we provide several results regarding the structure of derived categories of (nested) Hilbert schemes of points. We show that the criteria of Krug--Sosna and Addington for the universal ideal sheaf functor to be fully faithful resp.~a $\mathbb{P}$-functor are sharp. Then we show how to embed multiple copies of the derived category of the surface using these fully faithful functors. We also give a semiorthogonal decomposition for the nested Hilbert scheme of points on a surface, and finally we give an elementary proof of a semiorthogonal decomposition due to Toda for the symmetric product of a curve.
\end{abstract}


\section{Introduction}
Hilbert schemes of points on surfaces are a classical object of study, providing very explicit moduli spaces of sheaves with interesting links to resolution of singularities and representation theory.

In the past two decades their derived categories have been studied thoroughly, especially in the context of the derived McKay correspondence \cite{MR1824990,MR1839919}. They provide an important source of interesting behaviour for derived categories of smooth projective varieties, such as the construction of non-standard autoequivalences \cite{MR3477955,MR3455879,krug-nakajima} or interesting fully faithful functors \cite{MR3397451,MR3950704}. A good understanding of their derived categories, and functors relating them, has e.g.~led to a more abstract interpretation of their deformation theory \cite{MR3950704}.

In this paper we give several different results regarding the derived categories of Hilbert schemes of points on surfaces, and closely related varieties, namely nested Hilbert schemes and symmetric products of curves.

\paragraph{Converses}
First we will give two results which provide converses to important criteria in the literature, which are closely related to understanding the deformation theory of Hilbert schemes via fully faithful functors, and the construction of new auto-equivalences. The first result we discuss provides a converse to the fully faithfulness criterion \cite[theorem~1.2]{MR3397451}, where~$S$ is a smooth projective surface and~$\hilbn{n}{S}$ the associated Hilbert scheme of~$n$ points. Throughout the article we will denote~$\FF$ for the Fourier--Mukai functor~$\Phi_{\mathcal{I}}$, where~$\mathcal{I}$ is the ideal sheaf for the universal subscheme on~$S\times\hilbn{n}{S}$.
\begin{alphatheorem}
  \label{theorem:krug-sosna-converse}
  Let~$S$ be a smooth projective surface, and~$n\geq 2$. Then
  \begin{equation}
    \FF\colon\derived^\bounded(S)\to\derived^\bounded(\hilbn{n}{S})
  \end{equation}
  is fully faithful if and only if~$\HH^1(S,\mathcal{O}_S)=\HH^2(S,\mathcal{O}_S)=0$.
\end{alphatheorem}
This strengthens \cite[theorem~1.2]{MR3397451} from an if to an if and only if. From now on we will also say that~$\mathcal{O}_S$ is an \emph{exceptional} object, if the vanishing~$\HH^1(S,\mathcal{O}_S)=\HH^2(S,\mathcal{O}_S)=0$ holds.

The second result is a converse to the criterion of Addington \cite[theorem~3.1]{MR3477955}. This result motivated the definition of a~$\mathbb{P}^n$\dash functor, which is an interesting way of obtaining auto-equivalences of a variety.
\begin{alphatheorem}
  \label{theorem:addington-converse}
  Let~$S$ be a smooth projective surface, and~$n\geq 2$. Then
  \begin{equation}
    \FF\colon\derived^\bounded(S)\to\derived^\bounded(\hilbn{n}{S})
  \end{equation}
  is either a~$\mathbb{P}^m$\dash functor for some~$m\geq 1$ or a spherical functor if and only if~$S$ is a K3 surface (in which case it is a~$\mathbb{P}^m$\dash functor for~$m=n-1$).
\end{alphatheorem}
This strengthens \cite[theorem~3.1]{MR3477955}, which shows~$\FF$ is a~$\mathbb{P}^{n-1}$\dash functor if~$S$ is a~K3~surface, from an if to an if and only if. Remark that for any surface there are~$\mathbb{P}^{n-1}$\dash functors~$\derived^\bounded(S)\to\derived^\bounded(\hilbn{n}{S})$, but they are constructed differently, see \cref{remark:Pn-functors-abound} for more information.

\paragraph{Multiple copies} The third result gives an \emph{extension} of the fully faithfulness result of \cref{theorem:krug-sosna-converse}. Given a line bundle $L\in \Pic(S)$, there is a natural way to get an associated line bundle $\mathcal{D}_L\in \Pic(\hilbn{n}{S})$; see \cite{MR0335512} or \cref{subsection:hilbert} for details.

\begin{alphatheorem}
  \label{corollary:sod-surfaces}
  Let~$n\geq 3$. Let~$S$ be a smooth projective surface, such that~$\mathcal{O}_S$ is exceptional. Let
  \begin{equation}
    \derived^\bounded(S)
    =
    \left\langle
      L_1,\ldots,L_m,\mathcal{A}
    \right\rangle
  \end{equation}
  be a semiorthogonal decomposition, with~$L_i$ line bundles and~$\mathcal{A}$ some (possibly empty) complement to the exceptional collection. Then we have an induced semiorthogonal decomposition
  \begin{equation}
    \label{equation:sod-multiple}
    \derived^\bounded(\hilbn{n}{S})
    =
    \left\langle
      \FF(\derived^\bounded(S))\otimes \mathcal{D}_{L_1},
      \ldots,
      \FF(\derived^\bounded(S))\otimes \mathcal{D}_{L_m},
      \mathcal{B}
    \right\rangle
  \end{equation}
  where~$\mathcal{B}$ is defined as the complement.
\end{alphatheorem}

\paragraph{Symmetric powers of curves} Next we give an elementary proof of a recent result of Toda which describes the derived category of the symmetric powers~$C^{(n)}$ of a curve~$C$ \cite[corollary~5.11]{1805.00183v2}. Here we rely on the classical geometry of the Abel--Jacobi map and Schwarzenberger's description of the symmetric powers, together with a recent general result in homological projective geometry due to Jiang--Leung \cite{1811.12525v2} generalising Orlov's projective bundle formula to not necessarily locally free sheaves\footnote{Jiang--Leung have independently included a proof of \cref{theorem:toda} as an application of their projective bundle formula in the second version of their preprint. It is now the content of \cite[corollary~3.8]{1811.12525v2}.}.

\begin{alphatheorem}[Toda]
  \label{theorem:toda}
  For $n=g,\ldots,2g-2$, there exists a semiorthogonal decomposition
  \begin{equation}
    \label{equation:symsod}
    \derived^\bounded(C^{(n)})
    =
    \big\langle
      \underbrace{
        \derived^\bounded(\Jac C),
        \ldots,
        \derived^\bounded(\Jac C)
      }_{n-g+1},
      \derived^\bounded(C^{(2g-2-n)})
    \big\rangle.
  \end{equation}
\end{alphatheorem}
Here~$n=g,\ldots,2g-2$ constitutes the non-trivial range. For~$n\geq 2g-1$ the symmetric power is the projectivisation of a locally free sheaf. Hence there is a semiorthogonal decomposition of~$\derived^\bounded(C^{(n)})$ with all components equivalent to~$\derived^\bounded(\Jac C)$, and no contribution from lower-dimensional symmetric powers of the curve. For~$n\leq g-1$ the derived category is expected to be indecomposable. So \eqref{equation:symsod} would be a decomposition into indecomposable pieces. For the proof of \cref{theorem:toda} and more context, see \cref{section:sod-toda}.

\paragraph{Nested Hilbert schemes}
Finally we describe the derived categories of \emph{nested} Hilbert schemes of points. These are usually considered as tools to study the geometry of Hilbert schemes of points, but it turns out that their derived categories also exhibit an interesting phenomenon.

The proof of this decomposition follows in a straightforward way from the above-mentioned recent result of Jiang--Leung, and we have included it mostly to exhibit a large family of situations in which the Jiang--Leung result can be applied. For more details and an explicit description of the functors, see \cref{section:sod-nested}.

\begin{alphatheorem}
  \label{theorem:sod-nested}
  Let~$S$ be a smooth projective surface. Then we have a semiorthogonal decomposition
  \begin{equation}
    \derived^\bounded(S^{[n-1,n]})
    =
    \left\langle
      \derived^\bounded(S\times S^{[n-1]}),
      \derived^\bounded(S^{[n-2,n-1]})
    \right\rangle.
  \end{equation}
\end{alphatheorem}
Hence by induction one can obtain the following corollary.
\begin{alphacorollary}
  Let~$S$ be a smooth projective surface. Then for~$n\geq 2$ we have a semiorthogonal decomposition
  \begin{equation}
    \label{equation:corollary-nested}
    \derived^\bounded(S^{[n-1,n]})
    =
    \left\langle
      \derived^\bounded(S\times S^{[n-1]}),
      \derived^\bounded(S\times S^{[n-2]}),
      \ldots,
      \derived^\bounded(S)
    \right\rangle
  \end{equation}
\end{alphacorollary}
There are no cohomological conditions on the surface, and e.g.~when~$S$ is a K3~surface or an abelian surface, the decomposition \eqref{equation:corollary-nested} is in terms of indecomposable derived categories.

\paragraph{Acknowledgements}
We would like to thank Nick Addington, Theo Raedschelders and Lie Fu for interesting discussions. In particular, \cref{theorem:krug-sosna-converse} answers a question posed by Lie Fu.

We also want to thank Qingyuan Jiang and Naichung Conan Leung for informing us about their proof of \cref{theorem:toda} in the second version of their preprint.

The first author was supported by the Max Planck Institute for Mathematics and the University of Bonn.
The second author was supported by the University of Marburg.

\section{Preliminaries}
\label{section:preliminaries}
In this section we briefly recall some of the notation and constructions we will use throughout this article.

We denote~$k$ for the field~$\mathbb{C}$ of complex numbers. Some of the results in the literature that we are referring to are only proven using complex geometric methods, but we expect that all results in this article are valid for an algebraically closed field of characteristic~0.

We will abbreviate the bounded derived category of coherent sheaves~$\derived^\bounded(\coh X)$ on a smooth projective variety to~$\derived^\bounded(X)$. If~$X$ and~$Y$ are smooth projective varieties, and~$\mathcal{P}\in\derived^\bounded(X\times Y)$, then the associated \emph{Fourier--Mukai transform} is defined as
\begin{equation}
  \FM_{\mathcal{P}}=\FM_{\mathcal{P}}^{X\to Y}\colon\derived^\bounded(X)\to\derived^\bounded(Y):\mathcal{E}\mapsto\RRR q_*(p^*\mathcal{E}\otimes^\LLL\mathcal{P})
\end{equation}
where~$p$ and~$q$ denote the projections
\begin{equation}
  \begin{tikzcd}
    & X\times Y \arrow[ld, swap, "p"] \arrow[rd, "q"] \\
    X & & Y.
  \end{tikzcd}
\end{equation}

The left (resp.~right) adjoint of a Fourier--Mukai transform~$\Phi_{\mathcal{P}}$ is again a Fourier--Mukai transform, with kernel
\begin{equation}
  \mathcal{P}^{\mathrm{L}}\coloneqq\mathcal{P}^\vee\otimes^\LLL q^*\omega_Y[\dim Y]
\end{equation}
resp.
\begin{equation}
  \mathcal{P}^{\mathrm{R}}\coloneqq\mathcal{P}^\vee\otimes^\LLL p^*\omega_X[\dim X].
\end{equation}

\subsection{Induced sheaves and objects on Hilbert schemes}
\label{subsection:hilbert}
Throughout the article~$S$ will be a smooth projective surface, and~$\hilbn{n}{S}$ denotes the Hilbert scheme of~$n$ points on~$S$, which is a smooth projective variety of dimension $2n$; see \cite[theorem~2.4]{MR0237496}.

Let~$Z_n\subset S\times \hilbn{n}{S}$ be the universal family of length $n$ subschemes. Then we have a short exact sequence
\begin{equation}
  0\to\mathcal{I}_{Z_n}\to \mathcal{O}_{S\times \hilbn{n}{S}}\to \mathcal{O}_{Z_n}\to 0
\end{equation}
on~$S\times \hilbn{n}{S}$. We will denote the induced exact triangle of Fourier--Mukai transforms by
\begin{align}
  \label{eq:triangle}
  \FF\to \FF^{\prime}\to \FF^{\prime\prime}\to\FF[1],
\end{align}
where
\begin{itemize}
  \item $\FF=\FM_{\mathcal{I}_{Z_n}}$;
  \item $\FF^{\prime}=\FM_{\mathcal{O}_{S\times \hilbn{n}{S}}}\cong \HH^\bullet(S,-)\otimes_k \mathcal{O}_{\hilbn{n}{S}}$;
  \item $\FF^{\prime\prime}=\FM_{\mathcal{O}_{Z_n}}$.
\end{itemize}
In the literature, the objects in the image of $\FF^{\prime\prime}$ are called \emph{tautological objects} and denoted by $B^{[n]}\coloneqq\FF^{\prime\prime}(B)$ for $B\in \derived^\bounded(S)$.


Let $S^{(n)}\coloneqq S^n/\sym_n$ denote the symmetric product, let $\pi\colon S^n\to S^{(n)}$ be the quotient morphism, and let $\mu\colon S^{[n]}\to S^{(n)}$ be the Hilbert--Chow morphism.
By \cite[proposition~4.6]{MR0335512}, there is an injective group homomorphism
\begin{equation}
 \Pic(S)\to \Pic(\hilbn{n}{S})\quad,\quad L\mapsto \mathcal{D}_L\coloneqq\mu^*\bigl((\pi_* L^{\boxtimes n})^{\sym_n} \bigr)\,.
\end{equation}

\subsection{Formulae for Hom-spaces}
The results summarised in this subsection will be used in \cref{sect:KSconverse} and \cref{section:multiple-copies}.

By \cite{MR3205591}, we have the following formulae for the graded Hom-spaces between the above objects on the Hilbert scheme:
\begin{align}
  \Hom^\bullet(\FF^{\prime}A\otimes \mathcal{D}_L,\FF^{\prime}B \otimes \mathcal{D}_M)&\cong\HH^\bullet(S,A)^\vee\otimes \HH^\bullet(S,B)\otimes\Sym^n\Hom^\bullet(L,M) \label{eq:D1} \\
  \Hom^\bullet(\FF^{\prime}A \otimes \mathcal{D}_L,\FF^{\prime\prime}B \otimes \mathcal{D}_M)&\cong\HH^\bullet(S,A)^\vee\otimes \Hom^\bullet(L,B\otimes M)\otimes\Sym^{n-1}\Hom^\bullet(L,M) \label{eq:D2} \\
  \Hom^\bullet(\FF^{\prime\prime}A \otimes \mathcal{D}_L,\FF^{\prime}B \otimes \mathcal{D}_M)&\cong\Hom^\bullet(A\otimes L, M)\otimes \HH^\bullet(S,B)\otimes\Sym^{n-1}\Hom^\bullet(L,M) \label{eq:D3} \\
  \Hom^\bullet(\FF^{\prime\prime}A \otimes \mathcal{D}_L,\FF^{\prime\prime}B \otimes \mathcal{D}_M)&\cong\left( \Hom^\bullet(A\otimes L,B\otimes M)\otimes\Sym^{n-1}\Hom^\bullet(L,M) \right) \label{eq:D4} \\
  &\mkern-200mu \oplus\left( \Hom^\bullet(A\otimes L,M)\otimes \Hom^\bullet(L,B\otimes M)\otimes\Sym^{n-2}\Hom^\bullet(L,M) \right). \nonumber
\end{align}
For \eqref{eq:D3} and \eqref{eq:D4}, see \cite[theorem~3.17]{MR3205591}. For \eqref{eq:D1}, see \cite[remark~3.21]{MR3205591}. For \eqref{eq:D2}, see \cite[remark~3.20]{MR3205591}, which is a straight-forward generalisation of \cite[corollary~35]{MR2481854}. Here,~$A^\vee$ stands for the derived dual and~$\HH^\bullet(S,A)^\vee$ stands for the graded dual of the graded vector space $\HH^\bullet(S,A)\coloneqq\bigoplus_{i\in\mathbb{Z}}\HH^i(S,A)[-i]$. This means that if, for example, $\HH^\bullet(S,A)$ is concentrated in degrees~$0$, $1$, and~$2$, the dual~$\HH^\bullet(S,A)^\vee$ is concentrated in degrees~$-2$, $-1$, and~$0$. Furthermore, $\Sym^n\HH^\bullet(S,\mathcal{O}_S)$ denotes the symmetric power in the graded sense, which means that
\begin{equation}
  \Sym^n\HH^\bullet(S,\mathcal{O}_S)=\bigoplus_{i+j=n}\Sym^i\HH^{\mathrm{even}}(S,\mathcal{O}_S)\otimes_k \bigwedge\nolimits^j \HH^{\mathrm{odd}}(S,\mathcal{O}_S).
\end{equation}
We have~$\mathcal{D}_{\mathcal{O}_S}\cong \mathcal{O}_{\hilbn{n}{S}}$. Hence by setting~$L=\mathcal{O}_S=M$, the formulae \eqref{eq:D1}--\eqref{eq:D4} specialise to the following, as summarised in \cite[theorem~6.1]{MR3391883}:
\begin{align}
  \Hom^\bullet(\FF^{\prime}A,\FF^{\prime}B)&\cong\HH^\bullet(S,A)^\vee\otimes \HH^\bullet(S,B)\otimes\Sym^n\HH^\bullet(S,\mathcal{O}_S) \label{eq:1} \\
  \Hom^\bullet(\FF^{\prime}A,\FF^{\prime\prime}B)&\cong\HH^\bullet(S,A)^\vee\otimes \HH^\bullet(S,B)\otimes\Sym^{n-1}\HH^\bullet(S,\mathcal{O}_S) \label{eq:2} \\
  \Hom^\bullet(\FF^{\prime\prime}A,\FF^{\prime}B)&\cong\HH^\bullet(S,A^\vee)\otimes \HH^\bullet(S,B)\otimes\Sym^{n-1}\HH^\bullet(S,\mathcal{O}_S) \label{eq:3} \\
  \Hom^\bullet(\FF^{\prime\prime}A,\FF^{\prime\prime}B)&\cong\left( \Hom^\bullet(A,B)\otimes\Sym^{n-1}\HH^\bullet(S,\mathcal{O}_S) \right) \label{eq:4} \\
  &\mkern-100mu \oplus\left( \HH^\bullet(S,A^\vee)\otimes \HH^\bullet(S,B)\otimes\Sym^{n-2}\HH^\bullet(S,\mathcal{O}_S) \right)\,. \nonumber
\end{align}

\subsection{Equivariant sheaves and the Bridgeland--King--Reid--Haiman equivalence}
\label{subsection:bkr-h}
The results summarised in this subsection will be used in \cref{sect:Addconverse}.

Let~$G$ be a finite group acting on a smooth projective variety~$X$.
In our applications, the group~$G$ will be the symmetric group~$\sym_n$.
We will recall a few facts about (the category of)~$G$\dash equivariant sheaves on~$X$, and its derived category. For further details, we refer to \cite[section~4]{MR1824990}, \cite[section~2.2]{MR3788855}. A \emph{$G$\dash equivariant} sheaf on $X$ is a pair~$(E,\lambda)$ where~$E$ is a coherent sheaf on~$X$ and~$(\lambda_g\colon E\xrightarrow{\sim} g^*E)_{g\in G}$ is a \emph{$G$\dash linearisation}, i.e.\ a collection of isomorphisms such that for every $h,g\in G$ the following diagram commutes:
\begin{equation}
  \begin{tikzcd}
    E \arrow[r, "\lambda_g"] \arrow[rrr, bend right, "\lambda_{hg}"] & g^*(E) \arrow[r, "g^*\lambda_h"] & g^*\circ h^*(E) \arrow[r, "\simeq"] & (h\circ g)^*(E)
  \end{tikzcd}
\end{equation}
We denote the abelian category of~$G$\dash equivariant sheaves by~$\coh_G X$ and its bounded derived category by~$\derived^\bounded_G(X)\coloneqq\derived^\bounded(\coh_G X)$. Equivariant shaves can be canonically identified with coherent sheaves on the corresponding quotient stack~$[X/G]$, i.e.\ we have equivalences~$\coh_G X\cong \coh [X/G]$ and~$\derived^\bounded_G(X)\cong \derived^\bounded([X/G])$.

If~$H\le G$ is a subgroup we have the functor
\begin{equation}
  \Res^H_G\colon \coh_G X\to \coh_H X
\end{equation}
restricting $G$\dash linearisations to $H$\dash linearisations.

Let~$\chi$ be a one-dimensional representation of~$G$, which we identify with a character~$\chi\colon G\to k^\times$. Then there is an induced autoequivalence
\begin{equation}
  \MM_\chi\colon\coh_G X\to \coh_G X
\end{equation}
sending an equivariant sheaf~$(E,\lambda)$ to~$(E,\lambda^\chi)$, where~$\lambda_g^\chi=\chi(g)\cdot \lambda_g$. We will later consider the special case~$\MM_{\alt_2}$, where~$\alt_2$ is the unique non-trivial character of~$\sym_2$.

If the group~$G$ acts trivially on~$X$, a~$G$\dash linearisation of a sheaf~$E$ is the same as a~$G$\dash action on the sheaf. Hence in this case every~$G$\dash equivariant sheaf~$(E,\lambda)$ has a well-defined subsheaf of~$G$\dash invariants~$E^G\subset E$. This defines a functor
\begin{equation}
  (-)^G\colon \coh_G X\to \coh X.
\end{equation}

All three of the functors~$\Res$, $\MM_\chi$, and $(-)^G$ are exact, hence they induce functors between the equivariant derived categories.

Let now~$G$ act on two smooth projective varieties~$X$ and~$Y$. Then every~$G$\dash equivariant morphism~$f\colon X\to Y$ induces a pullback~$f^*\colon \coh_G Y\to \coh_G X$,
as well as a derived version $\LLL f^*\colon \derived^\bounded_G(Y)\to \derived^\bounded_G(X)$.
A special case that we will consider later is the following. Let~$S$ be a smooth projective surface. We consider~$\sym_n$ acting trivially on~$S$ and by permutation of the factors on~$S^n$. Then, the embedding of the small diagonal
\begin{equation}
  \delta\colon S\hookrightarrow S^n:x\mapsto (x,\dots,x)
\end{equation}
is~$\sym_n$\dash equivariant. Hence we get a pullback~$\LLL\delta^*\colon \derived^\bounded_{\sym_n}(S^n)\to \derived^\bounded(S)$.

To define the derived McKay correspondence for Hilbert schemes of points on surface, we consider the diagram
\begin{equation}
  \begin{tikzcd}
    \mathrm{I}^nS \arrow[r, "p"] \arrow[d, swap, "q"] & S^n \arrow[d, "\pi"] \\
    \hilbn{n}{S} \arrow[r, "\mu"] & S^{(n)}
  \end{tikzcd}
\end{equation}
where~$\mathrm{I}^nS\coloneqq(\hilbn{n}{S}\times_{S^{(n)}}S^n)_{\mathrm{red}}$ is the reduced fibre product, also called the \emph{isospectral Hilbert scheme}, and~$p$ and~$q$ are the projections.
With this notation,
\begin{equation}
  \Psi\coloneqq (-)^{\sym_n}\circ \RRR q_*\circ \LLL p^* \colon \derived^\bounded_{\sym_n}(S^n)\to \derived^\bounded(\hilbn{n}{S})
\end{equation}
is an equivalence, called \emph{derived McKay correspondence}; see \cite{MR1824990}, \cite{MR1839919}, \cite[proposition~2.8]{MR3788855}.

\section{A converse to the fully faithfulness criterion of Krug--Sosna}\label{sect:KSconverse}
In this section we prove \cref{theorem:krug-sosna-converse}. As pointed out in the introduction, one direction has been proven in \cite[theorem~1.2]{MR3397451}, so it suffices to show that if~$\FF$ is fully faithful, then~$\HH^1(S,\mathcal{O}_S)=\HH^2(S,\mathcal{O}_S)=0$.

We will denote~$q\coloneqq\hh^1(S,\mathcal{O}_S)$ and~$p_g\coloneqq\hh^2(S,\mathcal{O}_S)$, so that we have to show that~$p_g=q=0$. This is done by exhibiting objects~$A,B$ in~$\derived^\bounded(S)$ for which we can compute~$\Ext_{\hilbn{n}{S}}^i(\FF A,\FF B)$ explicitly and compare it to~$\Ext_S^i(A,B)$, to obtain a contradiction unless~$p_g=q=0$. We do this in the following lemmas.

If~$V^\bullet$ is a graded vector space, then we will denote
\begin{equation}
  \topdeg(V^\bullet)\coloneqq\max\{i\in\mathbb{Z}\mid V^i\neq0\}.
\end{equation}

\begin{lemma}
  Assume that~$\FF$ is fully faithful. Then~$p_g=0$.
\end{lemma}

\begin{proof}
  Assume on the contrary that~$p_g\neq 0$. This implies that~$\topdeg \Sym^i\HH^\bullet(S,\mathcal{O}_S)=2i$ for all~$i\in\mathbb{N}$.

  We set~$A$ to be a line bundle on~$S$ such that~$\HH^\bullet(S,A^\vee)$ is concentrated in degree 0 (we can take~$A^\vee$ to be a sufficiently high power of an ample line bundle), and~$B\coloneqq k(x)$ to be the skyscraper sheaf of some point~$x\in S$. Then we have
  \begin{equation}
    \left.
    \begin{array}{r}
      \topdeg \HH^\bullet(S,B) \\
      \topdeg \HH^\bullet(S,A)^\vee \\
      \topdeg \HH^\bullet(S,A^\vee) \\
      \topdeg \Hom^\bullet(A,B)
    \end{array}
    \right\}
    =0
  \end{equation}
  Using \eqref{eq:1}--\eqref{eq:4}, this implies
  \begin{equation}
    \begin{aligned}
      \topdeg \Hom^\bullet(\FF^{\prime}A,\FF^{\prime}B)&=2n \\
      \left.
      \begin{array}{c}
        \topdeg \Hom^\bullet(\FF^{\prime}A,\FF^{\prime\prime}B) \\
        \topdeg \Hom^\bullet(\FF^{\prime\prime}A,\FF^{\prime}B) \\
        \topdeg \Hom^\bullet(\FF^{\prime\prime}A,\FF^{\prime\prime}B)
      \end{array}
      \right\}
      &=2(n-1).
    \end{aligned}
  \end{equation}
  By the exact sequence
  \begin{equation}
    0=\Hom^{2n}(\FF^{\prime\prime}A,\FF^{\prime}B)\to \Hom^{2n}(\FF^{\prime}A, \FF^{\prime}B)\to \Hom^{2n}(\FF A, \FF^{\prime}B)
  \end{equation}
  we get that~$\Hom^{2n}(\FF A, \FF^{\prime}B)\neq 0$. On the other hand, the exact sequence
  \begin{equation}
    0=\Hom^{2n}(\FF^{\prime}A,\FF^{\prime\prime}B)\to \Hom^{2n}(\FF A,\FF^{\prime\prime}B)\to \Hom^{2n+1}(\FF^{\prime\prime}A,\FF^{\prime\prime}B)=0
  \end{equation}
  shows that~$\Hom^{2n}(\FF A,\FF^{\prime\prime}B)=0$. Hence the exact sequence
  \begin{equation}
    \Hom^{2n}(\FF A,\FF B)\to \Hom^{2n}(\FF A,\FF^{\prime}B)\to \Hom^{2n}(\FF A,\FF^{\prime\prime}B)=0
  \end{equation}
  gives~$\Hom^{2n}(\FF A,\FF B)\neq 0$. But if~$\FF$ were fully faithful, $\Hom^\bullet(\FF A,\FF B)\cong \Hom^\bullet(A,B)$ would be concentrated in degree~$0$. This shows that~$p_g=0$.
\end{proof}

We still need to show that~$q=0$.

\begin{lemma}
  Assume that~$\FF$ is fully faithful. Then~$q\leq n-1$.
\end{lemma}

\begin{proof}
  Assume on the contrary that~$q\geq n$. By the previous lemma we have that~$p_g=0$, hence~$\topdeg\Sym^i\HH^\bullet(S,\mathcal{O}_S)=i$ for all~$i\leq n$. Setting as above~$A$ to be a line bundle with~$\HH^\bullet(S,A^\vee)$ being concentrated in degree~$0$ and~$B=k(x)$, we get
  \begin{equation}
    \begin{aligned}
      \topdeg\Hom^\bullet(\FF^{\prime}A,\FF^{\prime}B)&=n \\
      \left.
      \begin{array}{r}
        \topdeg\Hom^\bullet(\FF^{\prime}A,\FF^{\prime\prime}B) \\
        \topdeg\Hom^\bullet(\FF^{\prime\prime}A,\FF^{\prime}B) \\
        \topdeg\Hom^\bullet(\FF^{\prime\prime}A,\FF^{\prime\prime}B)
      \end{array}
      \right\}
      &=n-1.
    \end{aligned}
  \end{equation}
  By the exact sequence
  \begin{equation}
    0=\Hom^{n}(\FF^{\prime\prime}A,\FF^{\prime}B)\to \Hom^{n}(\FF^{\prime}A, \FF^{\prime}B)\to \Hom^{n}(\FF A, \FF^{\prime}B)
  \end{equation}
  we get that~$\Hom^{n}(\FF A, \FF^{\prime}B)\neq 0$. On the other hand, the exact sequence
  \begin{equation}
    0=\Hom^{n}(\FF^{\prime}A,\FF^{\prime\prime}B)\to \Hom^{n}(\FF A,\FF^{\prime\prime}B)\to \Hom^{n+1}(\FF^{\prime\prime}A,\FF^{\prime\prime}B)=0
  \end{equation}
  shows that~$\Hom^{n}(\FF A,\FF^{\prime\prime}B)=0$. Hence the exact sequence
  \begin{equation}
    \Hom^{n}(\FF A,\FF B)\to \Hom^{n}(\FF A,\FF^{\prime}B)\to \Hom^{n}(\FF A,\FF^{\prime\prime}B)=0
  \end{equation}
  gives~$\Hom^{n}(\FF A,\FF B)\neq 0$. But if~$\FF$ were fully faithful, $\Hom^\bullet(\FF A,\FF B)\cong \Hom^\bullet(A,B)$ would be concentrated in degree~$0$. This shows that~$q\leq n-1$.
\end{proof}

We can rule out another case by similar methods as follows.

\begin{lemma}
  Assume that~$\FF$ is fully faithful. Then~$q\leq n-2$.
\end{lemma}

\begin{proof}
  By the previous two lemmas we have that~$p_g=0$ and that~$q\leq n-1$. Assume that~$q=n-1$. Then we have
  \begin{equation}
    \begin{aligned}
      \left.
        \begin{array}{c}
          \topdeg\Sym^n\HH^\bullet(S,\mathcal{O}_S) \\
          \topdeg\Sym^{n-1}\HH^\bullet(S,\mathcal{O}_S)
        \end{array}
      \right\}
      &=n-1 \\
      \topdeg\Sym^{n-2}\HH^\bullet(S,\mathcal{O}_S)&=n-2.
    \end{aligned}
  \end{equation}
  We set~$A=k(x)$ and~$B=k(y)$ to be skyscraper sheaves of two different points~$x,y$ of~$S$. Then
  \begin{equation}
    \begin{aligned}
      \topdeg\HH^\bullet(S,A^\vee)&=2 \\
      \left.
      \begin{array}{c}
        \topdeg\HH^\bullet(S,A)^\vee \\
        \topdeg\HH^\bullet(S,B)
      \end{array}
      \right\}
      &=0 \\
    \end{aligned}
  \end{equation}
  and
  \begin{equation}
    \Hom^\bullet(A,B)=0.
  \end{equation}
  Hence by \eqref{eq:1}--\eqref{eq:4},
  \begin{equation}
    \begin{aligned}
      \topdeg\Hom^\bullet(\FF^{\prime\prime}A,\FF^{\prime}B)&=n+1 \\
      \topdeg\Hom^\bullet(\FF^{\prime\prime}A,\FF^{\prime\prime}B)&=n \\
      \left.
      \begin{array}{c}
        \topdeg\Hom^\bullet(\FF^{\prime}A,\FF^{\prime}B) \\
        \topdeg\Hom^\bullet(\FF^{\prime}A,\FF^{\prime\prime}B)
      \end{array}
      \right\}
      &=n-1
    \end{aligned}
  \end{equation}
  The exact sequence
  \begin{equation}
    \Hom^{n+1}(\FF^{\prime\prime}A,\FF B)\to \Hom^{n+1}(\FF^{\prime\prime} A,\FF^{\prime}B)\to \Hom^{n+1}(\FF^{\prime\prime}A,\FF^{\prime\prime}B)=0
  \end{equation}
  shows that~$\Hom^{n+1}(\FF^{\prime\prime}A,\FF B)\neq0$. The exact sequence
  \begin{equation}
    0=\Hom^{n}(\FF^{\prime}A,\FF^{\prime\prime} B)\to \Hom^{n+1}(\FF^{\prime} A,\FF B)\to \Hom^{n+1}(\FF^{\prime}A,\FF^{\prime}B)=0
  \end{equation}
  gives~$\Hom^{n+1}(\FF^{\prime}A,\FF B)=0$. Hence the exact sequence
  \begin{equation}
    \Hom^{n}(\FF A,\FF B)\to \Hom^{n+1}(\FF^{\prime\prime} A,\FF B)\to \Hom^{n+1}(\FF^{\prime}A,\FF B)=0
  \end{equation}
  shows that~$\Hom^{n}(\FF A,\FF B)\neq 0$. But if~$\FF$ were fully faithful, we would have
  \begin{equation}
    \Hom^\bullet(\FF A,\FF B)\cong \Hom^\bullet(A,B)=0.
  \end{equation}
  This shows that~$q\leq n-2$.
\end{proof}

\begin{lemma}
  Assume that~$\FF$ is fully faithful. Then~$q=0$.
\end{lemma}

\begin{proof}
  By the previous lemmas we have that~$p_g=0$ and~$q\leq n-2$. Assume on the contrary that~$0<q\leq n-2$ (and~$n\geq 3$). Then we have
  \begin{equation}
    \left.
    \begin{array}{c}
      \topdeg\Sym^n\HH^\bullet(S,\mathcal{O}_S) \\
      \topdeg\Sym^{n-1}\HH^\bullet(S,\mathcal{O}_S) \\
      \topdeg\Sym^{n-2}\HH^\bullet(S,\mathcal{O}_S)
    \end{array}
    \right\}=q
  \end{equation}
  such that
  \begin{equation}
    (\Sym^n\HH^\bullet(S,\mathcal{O}_S))^q\cong(\Sym^{n-1}\HH^\bullet(S,\mathcal{O}_S))^q\cong (\Sym^{n-2}\HH^\bullet(S,\mathcal{O}_S))^q\cong k.
  \end{equation}
  We set~$A=B=k(x)$ to be the skyscraper sheaf of some point~$x\in S$ which gives
  \begin{equation}
    \begin{aligned}
      \left.
      \begin{array}{c}
        \topdeg\HH^\bullet(S,A^\vee) \\
        \topdeg\Hom^\bullet(A,B)
      \end{array}
      \right\}&=2 \\
\left.
\begin{array}{c}
        \topdeg\HH^\bullet(S,A)^\vee \\
        \topdeg\HH^\bullet(S,B)
      \end{array}
      \right\}&=0
            \end{aligned}
    \label{eq:top0}
  \end{equation}
  and
  \begin{equation}
    \HH^2(S,A^\vee)\cong\Hom^2(A,B)\cong \HH^0(S,B)\cong k
  \end{equation}
  Plugging this into \eqref{eq:3} and \eqref{eq:4} gives~$\Hom^{q+2}(\FF^{\prime\prime}A,\FF^{\prime}B)\cong k$ and~$\Hom^{q+2}(\FF^{\prime\prime}A,\FF^{\prime\prime}B)\cong k^2$. Hence the short exact sequence
  \begin{equation}
    \Hom^{q+2}(\FF^{\prime\prime}A,\FF^{\prime} B)\to \Hom^{q+2}(\FF^{\prime\prime} A,\FF^{\prime\prime}B)\to \Hom^{q+3}(\FF^{\prime\prime}A,\FF B)
  \end{equation}
  gives~$\Hom^{q+3}(\FF^{\prime\prime}A,\FF B)\neq 0$. Furthermore, combining \eqref{eq:top0} with \eqref{eq:1} and \eqref{eq:2} gives
  \begin{equation}
    \topdeg\Hom^\bullet(\FF^{\prime}A,\FF^{\prime}B)= \topdeg\Hom^\bullet(\FF^{\prime}A,\FF^{\prime\prime}B)=q.
  \end{equation}
  Hence the short exact sequence
  \begin{equation}
    0=\Hom^{q+2}(\FF^{\prime}A,\FF^{\prime\prime} B)\to \Hom^{q+3}(\FF^{\prime} A,\FF B)\to \Hom^{q+3}(\FF^{\prime}A,\FF^{\prime} B)=0
  \end{equation}
  shows~$\Hom^{q+3}(\FF^{\prime} A,\FF B)=0$. Using the exact sequence
  \begin{equation}
     \Hom^{q+2}(\FF A,\FF B)\to \Hom^{q+3}(\FF^{\prime\prime} A,\FF B)\to \Hom^{q+3}(\FF^{\prime}A,\FF B)=0
  \end{equation}
  we now get~$\Hom^{q+2}(\FF A,\FF B)\neq 0$. However, for~$q>0$ we have~$\Hom^{q+2}(A,B)= 0$, so~$\FF$ is not fully faithful. This contradiction shows that~$q=0$.
\end{proof}

\begin{remark}
  For higher-dimensional varieties the Hilbert scheme becomes (very) singular (unless~$n=2,3$), and it is better to consider symmetric quotient stacks. In this case there is no universal ideal sheaf, but one can define the truncated ideal sheaf \cite[\S5]{MR3397451} and it turns out that if~$\mathcal{O}_X$ is exceptional then the associated Fourier--Mukai transform is fully faithful \cite[proposition~5.6]{MR3397451}. It would be interesting to know whether this condition is also necessary.

  When~$n=2$ it is shown in \cite[theorem~A]{MR3950704} that~$\FF=\Phi_{\mathcal{I}}\colon\derived^\bounded(X)\to\derived^\bounded(\hilbn{2}{X})$ is fully faithful if~$\mathcal{O}_X$ is exceptional, and in fact this follows from the fully faithfulness of~$\derived^\bounded(X)\to\derived^\bounded([X^2/\sym_2])$, see also \cite[remark~11]{MR3950704}. It would be interesting to know whether this condition is also necessary.

  When~$n=3$ neither sufficiency nor necessity are known, and the question is closely related to the conjecture that~$\Psi\colon\derived^\bounded([X^3/\sym_3])\to\derived^\bounded(\hilbn{3}{X})$ is fully faithful.
\end{remark}

\section{A converse to the $\mathbb{P}^n$-functor criterion of Addington}
\label{sect:Addconverse}
In this section we prove \cref{theorem:addington-converse}. Let $X$ and $Y$ be smooth projective varieties. For a Fourier--Mukai transform~$F=\Phi_{\mathcal{P}}\colon \derived^\bounded(X)\to\derived^\bounded(Y)$ we write~$F^\LL=\Phi_{\mathcal{P}^\LL}$ resp.~$F^\RR=\Phi_{\mathcal{P}^\RR}\colon \derived^\bounded(X)\to\derived^\bounded(Y)$ for its left resp.~right adjoint functor; compare \cref{section:preliminaries}. The unit~$\eta\colon \id_{\derived^\bounded{X}} \to F^\RR\circ F$ of the adjunction~$F\dashv F^\RR$ is induced by a morphism~$\eta\colon \mathcal{O}_{\Delta_X}\to \mathcal{P}^\RR\star \mathcal{P}$ where $\mathcal{P}^\RR\star \mathcal{P}$ is the convolution product of the Fourier--Mukai kernels; see \cite{MR2964634} or \cite{MR2657369}. Hence one can define
\begin{equation}
  C\coloneqq\cone(\eta\colon\identity_{\derived^\bounded(X)}\to F^\RR\circ F)
\end{equation}
as the Fourier--Mukai transform along~$\cone\bigl(\mathcal{O}_{\Delta_X}\xrightarrow{\eta} \mathcal{P}^\RR\star \mathcal{P}\bigr)$.

We will need the following two notions.
\begin{definition}
  We say that a Fourier--Mukai functor~$F\colon\derived^\bounded(X)\to\derived^\bounded(Y)$ is \emph{spherical} if the associated \emph{cotwist} $C\coloneqq\cone(\eta\colon\identity_{\derived^\bounded(X)}\to F^\RR\circ F)$ is an auto-equivalence such that there is an isomorphism~$F^\RR\cong C\circ F^\LL$.
\end{definition}

\begin{definition}
  We say that a Fourier--Mukai functor~$F\colon\derived^\bounded(X)\to\derived^\bounded(Y)$ is a~\emph{$\mathbb{P}^n$\dash functor} if there exists an autoequivalence~$H$ of~$\derived^\bounded(X)$, called the \emph{$\mathbb{P}$-cotwist} of $F$, such that
  \begin{enumerate}
    \item $F^\RR\circ F\cong\identity_{\derived^\bounded(X)}\oplus\,H\oplus\ldots\oplus H^n$;

    \item the composition~$H\circ F^\RR\circ F\xrightarrow{i\circ F^\RR\circ F} F^\RR\circ F\circ R\circ F\xrightarrow{F^\RR\circ\epsilon\circ F} R\circ F$, where $i\colon H\to F^\RR\circ F$ is the embedding of the direct summand under the above isomorphism and~$\epsilon \colon F\circ F^\RR\to \id_{\derived^\bounded(Y)}$ is the counit of adjunction, is of the form
      \begin{equation}
        \begin{pmatrix}
          * & * & \ldots & * & * \\
          \identity_H & * & & * & * \\
          0 & \identity_H & & * & * \\
          \vdots & & \ddots & * & * \\
          0 & 0 & \ldots & \identity_H & *
        \end{pmatrix}
      \end{equation}
      when written in terms of the decomposition
      \begin{equation}
        H\oplus H^2\oplus\ldots\oplus H^{n+1}
        \to
        \identity_{\derived^\bounded(X)}\oplus\,H\oplus\ldots\oplus H^n;
      \end{equation}

    \item $F^\RR \cong H^n\circ F^\LL$.
  \end{enumerate}

\end{definition}

One important reason for the interest in spherical and~$\mathbb{P}$\dash functors is that they induce autoequivalences, called \emph{twists}, of their target categories; see \cite{MR2258045,MR3692883,MR3477955}.

Note that our assumption that~$X$ and~$Y$ are smooth and projective imply that~$\derived^\bounded(X)$ and~$\derived^\bounded(Y)$ possess Serre functors~$\serre_X=(-)\otimes \omega_X[\dim X]$ and~$\serre_Y=(-)\otimes \omega_Y[\dim Y]$.
The condition~$F^\RR\cong C\circ F^\LL$ for a spherical functor is equivalent to~$\serre_Y\circ F\circ C\cong F\circ \serre_X$, see \cite[page~225]{MR3477955}, and the condition~$F^\RR \cong H^n\circ F^\LL$ for a~$\mathbb{P}^n$-functor is equivalent to~$\serre_Y\circ F\circ H^n\cong F\circ \serre_X$; see \cite[page~246]{MR3477955}.
Hence we have the following property common to spherical functors and~$\mathbb{P}$\dash functors on which our proof of \cref{theorem:addington-converse} relies.
\begin{lemma}
  \label{lem:Pspherical}
  Let~$F\colon\derived^\bounded(X)\to\derived^\bounded(Y)$ be a Fourier--Mukai transform which is a spherical or a~$\mathbb{P}$\dash functor.
  Then there exists an autoequivalence~$D\in\Aut(\derived^\bounded(X))$ such that
  \begin{equation}
    \serre_Y\circ F\circ D\cong F\circ \serre_X\,.
  \end{equation}
  Here~$D$ is
  \begin{itemize}
    \item the cotwist $C$ of~$F$, if~$F$ is a spherical functor;
    \item the~$n$th power of the~$\mathbb{P}$\dash cotwist $H$ of~$F$, if~$F$ is a~$\mathbb{P}^n$\dash functor.
  \end{itemize}
\end{lemma}

The second ingredient is the observation that for a smooth projective surface~$S$ the functor~$\FF=\Phi_{\mathcal{I}_n}\colon \derived^\bounded(S)\to \derived^\bounded(\hilbn{n}{S})$ has a left inverse~$\II\colon \derived^\bounded(\hilbn{n}{S})\to \derived^\bounded(S) $ satisfying a compatibility with the Serre functors. This left-inverse is given by the composition
\begin{equation}
  \II\colon \derived^\bounded(\hilbn{n}{S})\xrightarrow{\Psi^{-1}} \derived^\bounded_{\sym_n}(S^n)\xrightarrow{\LLL \delta^*} \derived^\bounded_{\sym_n}(S)\xrightarrow{{\Res_{\sym_2}^{\sym_n}}} \derived^\bounded_{\sym_2}(S)\xrightarrow{\MM_{\alt_2}} \derived^\bounded_{\sym_2}(S)\xrightarrow{(-)^{\sym_2}} \derived^\bounded(S)
\end{equation}
where all functors are defined in \cref{subsection:bkr-h}.

\begin{lemma}
  \label{lem:G}
  The functor
  \begin{equation}
    \II\coloneqq(-)^{\sym_2}\circ{\MM_{\alt_2}}\circ{\Res_{\sym_2}^{\sym_n}}\circ\LLL\delta^*\circ \Psi^{-1}\colon \derived^\bounded(\hilbn{n}{S})\to \derived^\bounded(S)
  \end{equation}
  satisfies~$\II\circ \FF\cong \id_{\derived^\bounded(S)}$ and~$\II\circ \serre_{\hilbn{n}{S}}\cong \serre^n_S\circ\II$.
\end{lemma}

\begin{proof}
  The statement~$\II\circ \FF\cong \id$ is proved in \cite[theorem~3.6]{1808.05931v2}.

  For the compatibility with the Serre functors, first note that $\Psi^{-1}$, being an equivalence, commutes with the Serre functors, i.e.
  \begin{equation}
    \Psi^{-1}\circ\serre_{\hilbn{n}{S}}\cong\serre_{[\Sym^nS]}\circ\Psi^{-1}
  \end{equation}
  where~$\serre_{[\Sym^nS]}$ denotes the Serre functor of~$\derived^\bounded_{\sym_n}(S^n)$. It is given by~$\serre_{[\Sym^nS]}=-\otimes \omega_{S^n}[2n]$ where~$\omega_{S^n}$ is equipped with the natural~$\sym_n$\dash linearisation. This linearisation restricts to the trivial action on the pullback~$\delta^*\omega_{S^n}\cong \omega_S^{\otimes n}$. Hence for~$\derived^\bounded_{\sym_n}(S^n)$ we have natural isomorphisms
  \begin{equation}
    \begin{aligned}
      \bigl(\alt_2\otimes \Res_{\sym_2}^{\sym_n}\delta^*(E\otimes \omega_{S^n})\bigr)^{\sym_2}&\cong \bigl(\alt_2\otimes \Res_{\sym_2}^{\sym_n}\delta^*(E)\otimes \omega_{S}^{\otimes n}\bigr)^{\sym_2} \\
      &\cong \bigl(\alt_2\otimes \Res_{\sym_2}^{\sym_n}\delta^*(E)\bigr)^{\sym_2}\otimes \omega_{S}^{\otimes n} \,.
    \end{aligned}
  \end{equation}
  This gives
  \begin{equation}
    (-)^{\sym_2}\circ{\MM_{\alt_2}}\circ{\Res_{\sym_2}^{\sym_n}}\circ\delta^*\circ\serre_{[\Sym^nS]}\cong\serre_{S}^n\circ (-)^{\sym_2}\circ{\MM_{\alt_2}}\circ{\Res_{\sym_2}^{\sym_n}}\circ \delta^*\,.
  \end{equation}
\end{proof}

\begin{corollary}
  \label{cor:CS}
  Assume that~${\serre_{\hilbn{n}{S}}}\circ \FF\circ D\cong \FF\circ\serre_S$ for some~$D\in \Aut(\derived^\bounded(S))$. Then we have~$D\cong\serre_S^{-(n-1)}$.
\end{corollary}

\begin{proof}
  We postcompose both sides of the given isomorphism of functors by the left inverse~$\II$. By \cref{lem:G}, we get on the left-hand side
  \begin{equation}
    \II\circ\serre_{\hilbn{n}{S}}\circ \FF\circ D\cong\serre_S^n\circ D\,,
  \end{equation}
  while the right-hand side is $\II\circ \FF\circ \serre_S\cong \serre_S$. Now, postcomposing both sides with $\serre_S^{-n}$ gives the assertion~$D\cong\serre_S^{-(n-1)}$.
\end{proof}

For $x\in S$, we consider the subset $U_x\coloneqq\{\xi\in \hilbn{n}{S}\mid x\notin \xi\}\subseteq \hilbn{n}{S}$. It is an open subset whose complement is of codimension~$2$.

\begin{lemma}
  \label{lem:U}
  For~$x\in S$, we have~$\FF(k(x))|_{U_x}\cong \mathcal{O}_{U_x}$.
\end{lemma}

\begin{proof}
  We use the exact triangle of functors \eqref{eq:triangle}. By definition of~$\FF^{\prime\prime}$ as the Fourier--Mukai transform along the structure sheaf of the universal family, we have~$\FF^{\prime\prime}(k(x))|_{U_x}\cong 0$. Furthermore, as~$\HH^{\bullet}(S,k(x))\cong k[0]$, we have~$\FF^{\prime}(k(x))\cong \mathcal{O}_{\hilbn{n}{S}}$. The assertion follows from the exact triangle~$\FF\to \FF^{\prime}\to \FF^{\prime\prime}\to \FF[1]$.
\end{proof}

\begin{proposition}
  \label{prop:omegaO}
  Let~$\serre_{\hilbn{n}{S}}\circ \FF\circ D\cong \FF\circ \serre_S$ for some~$D\in \Aut(\derived^\bounded(S))$. Then~$\omega_S\cong \mathcal{O}_S$.
\end{proposition}

\begin{proof}
 We apply both sides of the given isomorphism to the skyscraper sheaf~$k(x)$ of some point $x\in S$. By \cref{cor:CS}, we have
 \begin{equation}
   D(k(x))\cong \serre_S^{-(n-1)}(k(x))\cong k(x)[-2(n-1)].
 \end{equation}
 Hence the given isomorphism yields an isomorphism~$\FF(k(x))\otimes \omega_{\hilbn{n}{S}}\cong \FF(k(x))$. Now, \cref{lem:U} gives~$\omega_{\hilbn{n}{S}}|_{U_x}\cong \mathcal{O}_{U_x}$. Since~$\hilbn{n}{S}$ is normal and the complement of~$U_x$ is of codimension~$2$, this implies $\mathcal{O}_{\hilbn{n}{S}}\cong\omega_{\hilbn{n}{S}}$.
Recall from \cref{subsection:hilbert} that there is an injective group homomorphism~$\Pic(S)\hookrightarrow \Pic(\hilbn{n}{S})$, $L\mapsto \mathcal{D}_L$ satisfying~$\mathcal{D}_{\mathcal{O}_S}\cong \mathcal{O}_{\hilbn{n}{S}}$ and~$\mathcal{D}_{\omega_S}\cong\omega_{\hilbn{n}{S}}$; for the latter isomorphism, see e.g.\ \cite[proposition~1.6]{MR2110899}. Hence~$\mathcal{O}_{\hilbn{n}{S}}\cong\omega_{\hilbn{n}{S}}$ implies $\mathcal{O}_S\cong \omega_S$.
\end{proof}

\begin{proof}[Proof of \cref{theorem:addington-converse}]
  Assume that~$\FF\colon \derived^\bounded(S)\to \derived^\bounded(\hilbn{n}{S})$ is a spherical functor or a~$\mathbb{P}^m$-functor for some $m\geq 1$. We want to show that~$S$ is a K3 surface (in which case it is known by \cite[theorem~3.1]{MR3477955} that~$\FF$ is a~$\mathbb{P}^{n-1}$\dash functor). By \cref{lem:Pspherical} together with \cref{prop:omegaO}, we see that~$\omega_S$ is trivial. Hence~$S$ is a K3 or an abelian surface. But for an abelian surface, we have~$\FF^\RR\circ\FF=\id_{\derived^\bounded(S)}\oplus[-1]^{\oplus 2}\oplus \dots \oplus [-(2n-1)]^{\oplus 2}\oplus [-2n]$ as computed in \cite[page~1199]{MR3391883}. Hence if~$S$ is an abelian surface, $\FF$ is neither a spherical nor a~$\mathbb{P}$\dash functor.
\end{proof}

\begin{remark}
  \label{remark:Pn-functors-abound}
  When~$S$ is a K3 surface the functor~$\FF$ was the first instance of a~$\mathbb{P}$\dash functor \cite[theorem~3.1]{MR3477955} from~$\derived^\bounded(S)$ to~$\derived^\bounded(\hilbn{n}{S})$. We now know that this is the only surface for which~$\FF$ gives such a functor. But there exist other constructions of~$\mathbb{P}$\dash functors $\derived^\bounded(S)$ to~$\derived^\bounded(\hilbn{n}{S})$ which work for arbitrary surfaces; see \cite{MR3455879, krug-nakajima}.
\end{remark}

\section{Embedding multiple copies}
\label{section:multiple-copies}
In this section we prove \cref{corollary:sod-surfaces}. The motivation for these results comes from a result for moduli of vector bundles on curves. In this context, the analogue of \cref{theorem:krug-sosna-converse} is the fully faithfulness of the Fourier--Mukai functor
\begin{equation}
  \Phi_{\mathcal{E}}\colon\derived^\bounded(C)\to\derived^\bounded(\moduli_C(r,\mathcal{L})),
\end{equation}
where~$\moduli_C(r,\mathcal{L})$ is the moduli space of rank~$r$ bundles with determinant~$\mathcal{L}$ and~$\mathcal{E}$ the universal vector bundle on~$C\times\moduli_C(r,\mathcal{L})$. Here~$\gcd(r,\deg\mathcal{L})=1$ so that the moduli space is smooth and projective of dimension~$(r^2-1)(g-1)$. This is shown for~$r=2$ in \cite{MR3764066,MR3713871} and~$r\geq 2$ in \cite{MR3954042} (under suitable conditions on the genus and with~$\deg\mathcal{L}=1$) and more generally in \cite{belmans-mukhopadhyay-work-in-progress}.

Now in \cite{MR3954042} it was observed that~$\Phi_{\mathcal{E}}$ can be twisted by~$\mathcal{O}_{\moduli_C(r,\mathcal{L})}(1)$, so that~$\Phi_\mathcal{E}(\derived^\bounded(C))$ and~$\Phi_{\mathcal{E}}(\derived^\bounded(C))\otimes\mathcal{O}_{\moduli_C(r,\mathcal{L})}(1)$ are semiorthogonal. This follows from~$\moduli_C(r,\mathcal{L})$ being a smooth projective Fano variety \emph{of index~2}.

Hilbert schemes of points on surfaces are never Fano\footnote{As the Hilbert--Chow morphism is a crepant resolution, the anticanonical bundle is never positive on the exceptional divisor.}, but nevertheless a similar method of embedding multiple copies exists, as will be shown in this section.

\begin{proposition}
  \label{prop:sorth}
 Let $S$ be a smooth projective surface, and let $L, M\in \Pic(S)$ with $\Hom^\bullet(L,M)=0$. Then, for $n\ge 3$, the subcategories $(\im\FF)\otimes L\subset \derived^\bounded(\hilbn{n}{S})$ and $(\im\FF)\otimes M \subset \derived^\bounded(\hilbn{n}{S})$ are semiorthogonal: for every $A,B\in \derived^\bounded(S)$, we have
  \begin{equation}
    \Hom^\bullet(\FF A\otimes \mathcal{D}_L, \FF B\otimes \mathcal{D}_M)=0\,.
  \end{equation}
\end{proposition}

\begin{proof}
  As $n\ge 3$, the vanishing $\Hom^\bullet(L,M)=0$ implies
  \begin{equation}
    \Sym^{n}\Hom^\bullet(L,M)= \Sym^{n-1}\Hom^\bullet(L,M)=\Sym^{n-2}\Hom^\bullet(L,M)=0\,.
  \end{equation}
  This means that all the Hom-spaces \eqref{eq:D1}--\eqref{eq:D4} vanish. Then, by the excact triangle \eqref{eq:triangle}, also $\Hom^\bullet(\FF A\otimes \mathcal{D}_L, \FF B\otimes \mathcal{D}_M)=0$.
\end{proof}

\begin{proof}[Proof of \cref{corollary:sod-surfaces}]
  Let $\mathcal{O}_S$ be exceptional. By \cref{theorem:krug-sosna-converse}, the functor $\FF\colon \derived^\bounded(S)\to \derived^\bounded (\hilbn{n}{S})$ is fully faithful. Let now $L_1,\dots, L_m$ be an exceptional collection of line bundles. Since tensor product by the associated line bundle $\mathcal{D}_{L_i}$ is an autoequivalence of $\derived^\bounded (\hilbn{n}{S})$, the functor $\FF(-)\otimes \mathcal{D}_{L_i}\colon \derived^\bounded(S)\to \derived^\bounded (\hilbn{n}{S})$ is again fully faithful for every $i=1,\dots,n$.
This means that the subcategories $\FF( \derived^\bounded(S))\otimes \mathcal{D}_{L_i}$ of $\derived^\bounded (\hilbn{n}{S})$ are admissible. The semiorthogonality of these subcategories is provided by \cref{prop:sorth}.
\end{proof}

\begin{remark}
  Even when~$\mathcal{A}$ is trivial, the category~$\mathcal{B}$ in \eqref{equation:sod-multiple} is not. To see this it suffices to observe that the number of copies of~$\derived^\bounded(S)$ in \eqref{equation:sod-multiple} does not grow with~$n$.
\end{remark}

\section{Semiorthogonal decompositions for symmetric products of curves}
\label{section:sod-toda}
In this section we prove \cref{theorem:toda}. The Hilbert scheme of points on a smooth projective curve is nothing but its symmetric power: $\hilbn nC\cong C^{(n)}$. And~$C^{(n)}$ has a description in terms of a projectivisation of a coherent but not necessary locally free sheaf. This is the content of \cite[theorem~4]{MR0151459}, which describes the Abel--Jacobi map
\begin{equation}
  C^{(n)}\to J\coloneqq\Jac C
\end{equation}
as a projectivisation of a coherent sheaf on $J$.

For~$i\geq 2g-1$ the Abel--Jacobi map has the structure of a projective bundle for a locally free sheaf, so one can just apply Orlov's projective bundle formula \cite[theorem~2.6]{MR1208153} to describe~$\derived^\bounded(C^{(n)})$.

For~$n\leq g-1$ it is expected that~$\derived^\bounded(C^{(n)})$ is indecomposable. This is proven for~$n=1$ in \cite{MR2838062} and for~$n\leq\lfloor\frac{g+3}{2}\rfloor$ in \cite{1807.10702v1,belmans-okawa-ricolfi}.

In the interesting range~$n=g,\ldots,2g-2$, where the Abel--Jacobi map is surjective but not a bundle, a semiorthogonal decomposition for~$\derived^\bounded(C^{(n)})$ is known, and forms the content of \cref{theorem:toda}. In \cite{1805.00183v2}, this is shown using wall-crossing methods. We will give a more elementary proof using the description of~$C^{(n)}$ as a projectivisation of a coherent sheaf on the Jacobian~$J=\Jac C$, together with the description of the derived category of such projectivisations \cite[theorem~3.4]{1811.12525v2}. Let us recall the statement of loc.\ cit.

For a coherent sheaf~$\mathcal{G}$ of rank~$r$ on~$X$ we will denote
\begin{equation}
  X^{>i}(\mathcal{G})\coloneqq\{x\in X\mid\rk\mathcal{G}(x)>i\}
\end{equation}
such that the \emph{singular locus}~$\Sing(\mathcal{G})$ of~$\mathcal{G}$ is~$X^{>r}(\mathcal{G})$, and the \emph{smooth part~$\Sing(\mathcal{G})^{\mathrm{sm}}$ of the singular locus} is~$X^{>r}(\mathcal{G})\setminus X^{>r+1}(\mathcal{G})$.

\begin{theorem}[Jiang--Leung's generalised projective bundle formula]
  \label{theorem:jiang-leung-bundle}
  Let~$X$ be a smooth projective variety. Let~$\mathcal{G}$ be a coherent sheaf of rank~$r$, which locally admits a 2-step locally free resolution. Assume that
  \begin{enumerate}
    \item $\mathbb{P}(\mathcal{G})$ is irreducible of expected dimension~$\dim X+r-1$;
    \item $\mathbb{P}(\sExt^1(\mathcal{G},\mathcal{O}_X))$ is irreducible of expected dimension~$\dim X-r-1$;
    \item the smooth part of the singular locus~$\Sing(\mathcal{G})^{\mathrm{sm}}$ is non-empty of expected codimension~$r+1$ in~$X$.
  \end{enumerate}
  Then
  \begin{enumerate}
    \item $\mathbb{P}(\sExt^1(\mathcal{G},\mathcal{O}_X))\to\Sing(\mathcal{G})$ is a resolution of singularities,
    \item there exists a semiorthogonal decomposition
      \begin{equation}
        \label{equation:generalised-bundle-sod}
        \derived^\bounded(\mathbb{P}(\mathcal{G}))
        =
        \big\langle
          \derived^\bounded(\mathbb{P}(\sExt^1(\mathcal{G},\mathcal{O}_X))),
          \underbrace{
            \derived^\bounded(X),
            \ldots
            \derived^\bounded(X)
          }_{r}
        \big\rangle.
      \end{equation}
  \end{enumerate}
\end{theorem}
The fully faithful functors in \eqref{equation:generalised-bundle-sod} are
\begin{equation}
  \RRR q_{1,*}\circ\LLL q_2^*\colon\derived^\bounded(\mathbb{P}(\sExt^1(\mathcal{G},\mathcal{O}_X)))\to\derived^\bounded(\mathbb{P}(\mathcal{G}))
\end{equation}
where
\begin{equation}
  \begin{tikzcd}
    & \mathbb{P}(\mathcal{G})\times_X\mathbb{P}(\sExt^1(\mathcal{G},\mathcal{O}_X)) \arrow[ld, swap, "q_1"] \arrow[rd, "q_2"] \\
    \mathbb{P}(\sExt^1(\mathcal{G},\mathcal{O}_X)) & & \mathbb{P}(\mathcal{G})
  \end{tikzcd}
\end{equation}
are the projections and
\begin{equation}
  \LLL\pi^*(-)\otimes\mathcal{O}_{\mathbb{P}(\mathcal{G})}(i)\colon\derived^\bounded(X)\to\derived^\bounded(\mathbb{P}(\mathcal{G}))
\end{equation}
for~$i=1,\ldots,r$, where
$\pi\colon\derived^\bounded(\mathbb{P}(\mathcal{G}))\to\derived^\bounded(X)$ is the natural morphism.

Later, we will also use a special case of \cref{theorem:jiang-leung-bundle} for the blowup in a singular center; see \cref{theorem:jiang-leung-blowup}.

\begin{remark}
  \label{rem:cond23}
  Under the assumption that $\mathcal G$ locally admits a $2$-step locally free resolution, we have
  \begin{equation}
    \im\bigl(\mathbb P(\sExt^1(\mathcal G,\mathcal{O}_X))\xrightarrow\pi X\bigr)= \Sing(\mathcal G)\,;
  \end{equation}
  see \cite[remark~3.5]{1811.12525v2}. We can replace condition~3 of \cref{theorem:jiang-leung-bundle} by the (under the presence of conditions 1 and 2) equivalent condition that~$\im(\pi)=\Sing(\mathcal G)$ has the expected dimension~$\dim X-r+1$ in $X$. Indeed, if~$\mathbb P(\sExt^1(\mathcal G,\mathcal{O}_X))$ and~$\im(\pi)$ are of the same dimension, the morphism~$\pi$ must be generically finite. But the fibers of~$\pi$ are projective spaces, hence, in particular, connected. That means that~$\pi$ is generically an isomorphism. The locus over which~$\pi$ is an isomrphism is exactly~$\Sing(\mathcal G)^{\mathrm{sm}}$; see again \cite[remark~3.5]{1811.12525v2}.
%
\end{remark}

Let us now recall some definitions and results from \cite{MR0151459}.
Fix a base point~$x_0\in C$ and, for~$n\in\mathbb{N}$, write~$\mathcal{O}_C(n)\coloneqq\mathcal{O}_C(n\cdot x_0)$.
Let~$\mathcal{P}\in \Pic(C\times J)$ be the Poincar\'e bundle, and let
\begin{equation}
  \begin{tikzcd}
    & C\times J \arrow[ld, swap, "p"] \arrow[rd, "q"] \\
    J & & C
  \end{tikzcd}
\end{equation}
be the projections. There are for every~$n\in \mathbb{N}$ the \emph{Picard sheaves}
\begin{equation}
  \left\{
    \begin{aligned}
      \mathcal{E}_n&\coloneqq p_*\bigl(\mathcal{P}\otimes q^*\mathcal{O}_C(n)\bigr) \\
      \mathcal{F}_n&\coloneqq\RR^1p_*\bigl(\mathcal{P}\otimes q^*\mathcal{O}_C(n)\bigr)\,.
    \end{aligned}
  \right.
\end{equation}
We can write the canonical bundle of~$C$ in the form~$\omega_X\cong\mathcal{O}_C(2g-2)\otimes \mathcal{K}$ for some degree zero line bundle~$\mathcal{K}\in \Pic^0(C)$. We consider the automorphism~$\theta$ of~$J=\Pic^0(C)$ given by~$\theta(\mathcal{L})\coloneqq\mathcal{K}\otimes \mathcal{L}^\vee$.
Then, for every~$n\geq 1$, there is an isomorphism
\begin{equation}
  \label{eq:SymP}
  C^{(n)}\cong\mathbb{P}(\theta^*\mathcal{F}_{2g-2-n})
\end{equation}
of varieties over~$J$; see \cite[theorem~4]{MR0151459}. Furthermore, we have isomorphisms
\begin{equation}
  \label{eq:thetaEF}
  \left\{
    \begin{aligned}
      \theta^*\mathcal{E}_n&\cong p_*\bigl(\mathcal{P}^\vee\otimes q^*(\mathcal{O}_C(n)\otimes \mathcal{K})\bigr) \\
      \theta^*\mathcal{F}_n&\cong \RR^1p_*\bigl(\mathcal{P}^\vee\otimes q^*(\mathcal{O}_C(n)\otimes \mathcal{K})\bigr)\,;
    \end{aligned}
  \right.
\end{equation}
see \cite[lemma~1]{MR0151459}. We will also use the fact that
\begin{equation}
  \label{eq:Evanish}
  \mathcal{E}_i=0\quad\text{for } i<g\,;
\end{equation}
see \cite[corollary~2]{MR0151459}.

\begin{lemma}
  \label{lem:Ext1}
  For~$n=g,\ldots,2g-2$ we have~$\sExt^1(\theta^*\mathcal{F}_{2g-2-n},\mathcal{O}_J)\cong \mathcal{F}_{n}$.
\end{lemma}

\begin{proof}
  Note that $2g-2-n<g$. Hence by \eqref{eq:Evanish} we have~$0=\mathcal{E}_{2g-2-n}$. Combining this vanishing with \eqref{eq:thetaEF} gives
  \begin{equation}
    0=\theta^*\mathcal{E}_{2g-2-n}\cong p_*\bigl(\mathcal{P}^\vee\otimes q^*(\mathcal{O}_C(2g-2-n)\otimes \mathcal{K})\bigr)\,.
  \end{equation}
  Hence the derived pushforward~$\RRR p_*\bigl(\mathcal{P}^\vee\otimes q^*(\mathcal{O}_C(2g-2-n)\otimes \mathcal{K})\bigr)$ is concentrated in degree 1, which means that
  \begin{equation}
    \RRR p_*\bigl(\mathcal{P}^\vee\otimes q^*(\mathcal{O}_C(2g-2-n)\otimes \mathcal{K})\bigr)\cong \theta^*\mathcal{F}_{2g-2-n}[-1]\,.
  \end{equation}
  This implies
  \begin{equation}\label{eq:Ext1}
   \sExt^1(\theta^*\mathcal{F}_{2g-2-n},\mathcal{O}_J)\cong \mathcal{H}^0\Bigr(\RRRsHom_J\bigl(\RRR p_*(\mathcal{P}^\vee\otimes q^*(\mathcal{O}_C(2g-2-n)\otimes \mathcal{K})),\mathcal{O}_J\bigl)\Bigr)
  \end{equation}
  By Grothendieck--Verdier duality, we have
  \begin{equation}\label{eq:GVdual}
   \RRRsHom_J\bigl(\RRR p_*(\mathcal{P}^\vee\otimes q^*(\mathcal{O}_C(2g-2-n)\otimes \mathcal{K})),\mathcal{O}_J\bigl)\cong\RRR p_*\RRRsHom_{C\times J}\bigl(\mathcal{P}^\vee\otimes q^*(\mathcal{O}_C(2g-2-n)\otimes \mathcal{K}),\omega_p\bigl)[1]
  \end{equation}
  Note that $\omega_p\cong q^*\omega_C\cong q^*(\mathcal{O}_C(2g-2)\otimes \mathcal{K})$. Hence
  \begin{equation}\label{eq:RHom}
    \RRRsHom_{C\times J}\bigl(\mathcal{P}^\vee\otimes q^*(\mathcal{O}_C(2g-2-n)\otimes \mathcal{K}),\omega_p\bigl)\cong \mathcal{P}\otimes q^*\mathcal{O}_C(n)[0]\,.
  \end{equation}
  Plugging \eqref{eq:RHom} into \eqref{eq:GVdual} gives
  \begin{equation}
    \label{eq:final}
    \mathcal{H}^0\Bigr( \RRRsHom_{J}\bigl(\RRR p_*(\mathcal{P}^\vee\otimes q^*(\mathcal{O}_C(2g-2-n)\otimes \mathcal{K})),\mathcal{O}_J\bigl)\Bigl)\cong\mathrm{R}^1p_*\bigl(\mathcal{P}\otimes q^*\mathcal{O}_C(n)\bigr)=\mathcal{F}_n
  \end{equation}
  Combining \eqref{eq:final} with \eqref{eq:Ext1} gives the assertion.
\end{proof}

\begin{proof}[Proof of \cref{theorem:toda}]
  Set~$\mathcal{G}\coloneqq\theta^*\mathcal{F}_{2g-2-n}$. By \eqref{eq:SymP} together with \cref{lem:Ext1}, we have
  \begin{equation}
    \mathbb{P}(\sExt^1(\mathcal{G},\mathcal{O}_J))\cong\mathbb{P}(\mathcal{F}_n)\cong C^{(2g-2-n)} .
  \end{equation}
  Hence the semiorthogonal decomposition \eqref{equation:generalised-bundle-sod} from \cref{theorem:jiang-leung-bundle} gives rise to \eqref{equation:symsod}, once we have checked that~$\mathcal{G}$ satisfies the assumptions. Note that the isomorphism $\mathbb{P}(\sExt^1(\mathcal{G},\mathcal{O}_J)) \cong C^{(2g-2-n)}$ identifies the canonical map $\pi \colon \mathbb{P}(\sExt^1(\mathcal{G},\mathcal{O}_J))\to J$ with $\theta\circ \AJ$, where $\AJ\colon C^{(2g-2-n)}\to J$ is the Abel--Jacobi map. In particular, $\im(\pi)=\theta(\im(\AJ))$.

  By \cite[proposition 4 and the following remark]{MR0151459}, we know that~$\mathcal{G}$ has a two-term resolution by locally free sheaves. Note that~$r\coloneqq\rk(\mathcal{G})=\dim C^{(n)}-\dim J+1=n-g+1$. Hence~$\mathbb{P}(\Ext^1(\mathcal{G},\mathcal{O}_J))\cong C^{(2g-2-n)}$ has the expected dimension~$2g-2-n=g-r-1$.

  Since~$2g-2-n< g$, we have $\dim(\im(\AJ))=\dim C^{(2g-2-n)}=2g-2-n$; see \cite[page~25]{MR0770932}.
  Since, as noted above we have $\im (\pi)=\theta(\im(\AJ))$, we also have $\dim(\im(\pi))=2g-2-n$. By \cref{rem:cond23}, this shows that also the third condition of \cref{theorem:jiang-leung-bundle} is fulfilled.
\end{proof}

\begin{remark}
  The same technique has been used independently by the authors of \cite{1811.12525v2} in the second version of their preprint.
\end{remark}

\section{Semiorthogonal decompositions for nested Hilbert schemes}
\label{section:sod-nested}
In this section we prove \cref{theorem:sod-nested}. Nested Hilbert schemes are usually seen as a means to set up correspondences between Hilbert schemes of points, e.g.~in the construction of actions of Heisenberg algebras on the cohomology of Hilbert schemes. But one can also study them for their own sake, as we do in this section.

\begin{definition}
  Let~$S$ be a smooth projective surface. Let~$n\geq 2$. The \emph{nested Hilbert scheme} is
  \begin{equation}
    \nestedhilbn{n-1}{n}{S}\coloneqq\big\{ (\zeta,\xi)\in\hilbn{n-1}{S}\times\hilbn{n}{S} \mid \zeta\subset\xi \big\}\hookrightarrow\hilbn{n-1}{S}\times\hilbn{n}{S}.
  \end{equation}
\end{definition}
By \cite[theorem~3.0.1]{MR1616606} the nested Hilbert scheme~$\nestedhilbn{n-1}{n}{S}$ is a smooth projective variety of dimension~$2n$, which is of Kodaira dimension~$n\kappa(S)$ by the geometric description \eqref{equation:nested-geometry}, together with \cite[theorem~11.1.2]{MR2665168}.

\begin{remark}
 In the literature, also nested Hilbert schemes of the form
$\nestedhilbn{m}{n}{S}\subset \hilbn{m}{S}\times\hilbn{n}{S}$
 for~$m<n$ are considered. However, for $m\neq n-1$ they are singular; see \cite[theorem~3.0.1]{MR1616606}.
Hence we will only consider the case~$m=n-1$.
%
\end{remark}

As summarised in \cite[\S3.1]{MR3477955} the nested Hilbert scheme~$\nestedhilbn{n-1}{n}{S}$ comes with two projection morphisms which we will denote as
\begin{equation}
  \label{equation:nested-geometry}
  \begin{tikzcd}
    & \nestedhilbn{n-1}{n}{S} \arrow[ld, swap, "f_n"] \arrow[rd, "g_n"] \\
    \hilbn{n-1}{S} & & \hilbn{n}{S}.
  \end{tikzcd}
\end{equation}
These send the pair~$(\zeta,\xi)$ to~$\zeta$ (resp.~$\xi$). Moreover, we have the morphism
\begin{equation}
  q_n\colon\nestedhilbn{n-1}{n}{S}\to S
\end{equation}
which sends the pair~$(\zeta,\xi)$ to the difference~$\xi\setminus\zeta$.
The morphism
\begin{equation}
 \phi_n\coloneqq q_n\times g_n\colon\nestedhilbn{n-1}{n}{S}\to S\times\hilbn{n}{S}
\end{equation}
has as its image the universal subscheme~$Z_n$, and is a resolution of singularities of~$Z_n$. Recall that~$Z_n$ is singular, as soon as~$n\geq 3$. However, it is always Cohen--Macaulay as it is flat and finite over the smooth base $S^{[n]}$. 
For~$n=1$ it is nothing but the diagonal, whilst for~$n=2$ it is~$\Bl_\Delta S\times S$.

The morphism
\begin{equation}
  \label{equation:nested-blowup}
  \gamma_n=q_n\times f_n\colon\nestedhilbn{n-1}{n}{S}\to S\times\hilbn{n-1}{S}
\end{equation}
on the other hand is the \emph{blowup} of~$S\times\hilbn{n-1}{S}$ in the universal subscheme~$Z_{n-1}$.

The proof of \cref{theorem:sod-nested} relies on a generalisation of Orlov's blowup formula \cite[theorem~4.3]{MR1208153}, which in this case cannot be applied as the center of the blowup in \eqref{equation:nested-blowup} is singular. This can be remedied by using a more general version of the blowup formula which is an instance of a projective bundle formula for not necessarily locally free sheaves in homological projective geometry \cite[theorem~3.4]{1811.12525v2}, see \cref{theorem:jiang-leung-bundle}. The special case we will apply is described in \cite[\S3.1.2]{1811.12525v2}.

\begin{theorem}[Jiang--Leung's generalised blowup formula]
  \label{theorem:jiang-leung-blowup}
  Let~$X$ be a smooth projective variety. Let~$Z\hookrightarrow X$ be a Cohen--Macaulay subscheme of codimension 2 which is cut out by the ideal sheaf~$\mathcal{I}_Z$. This sheaf admits a locally free resolution
  \begin{equation}
    0\to\mathcal{F}\overset\sigma\to\mathcal{E}\to\mathcal{I}_Z\to 0
  \end{equation}
  where~$\rk\mathcal{E}=\rk\mathcal{F}+1$.

  Then there exists a semiorthogonal decomposition
  \begin{equation}
    \derived^\bounded(\Bl_ZX)
    =
    \left\langle
      \derived^\bounded(\widetilde{Z}),
      \derived^\bounded(X)
    \right\rangle,
  \end{equation}
  where~$\widetilde{Z}\to Z$ is the (Springer-type) resolution of singularities given by
  \begin{equation}
    \widetilde{Z}
    \coloneqq \mathbb P(\sExt^1(\mathcal I_Z,\mathcal{O}_X))=
    \{
      (x,[H_x])\mid\image\sigma^\vee(x)\subseteq H_x
    \}
    \subseteq\mathbb{P}_X(\mathcal{F}^\vee)
  \end{equation}
\end{theorem}

With this result available to us, the proof of \cref{theorem:sod-nested} is very short.
\begin{proof}[Proof of \cref{theorem:sod-nested}]
  By the proof of \cite[lemma~4.7]{MR0335512} we have that~$Z_{n-1}$ is Cohen--Macaulay, and it is of codimension~2. So it suffices to check that the resolution $\mathbb P(\sExt^1(\mathcal I_{Z_{n-1}},\mathcal{O}_{S\times \hilbn{n}{S}}))$ of~$Z_{n-1}$ is isomorphic to~$\nestedhilbn{n-2}{n-1}{S}$. But this is the content of \cite[theorem~2]{MR1830227}.
\end{proof}


\bibliographystyle{plain}
\bibliography{automatic-clean,other}

\end{document}